\documentclass[11pt,reqno,a4paper]{amsart}
\usepackage{amssymb}
\usepackage{amsfonts}
\usepackage{amsmath, amsthm, amssymb}
\usepackage[english]{babel}
\usepackage[mathcal]{eucal}
\usepackage{xcolor}
\usepackage{float}

\title[]{On the asymptotic normality and the construction of confidence intervals for estimators after sampling with probabilistic and deterministic stopping rules}
\author{Ben Berckmoes and Geert Molenberghs}
\date{}
\thanks{Ben Berckmoes is post doctoral fellow at the Fund for Scientific Research of Flanders (FWO); financial support from the IAP research network \#P7/06 of the Belgian Government (Belgian Science Policy) is gratefully acknowledged.}

\begin{document}

\maketitle

\hyphenation{ex-ten-da-ble}

\newtheorem{pro}{Proposition}
\newtheorem{lem}{Lemma}
\newtheorem{thm}[pro]{Theorem}
\newtheorem{de}[pro]{Definition}
\newtheorem{co}[pro]{Comment}
\newtheorem{no}[pro]{Notation}
\newtheorem{vb}[pro]{Example}
\newtheorem{vbn}[pro]{Examples}
\newtheorem{gev}[pro]{Corollary}

\begin{abstract}
A key feature of a sequential study is that the actual sample size is a random variable that typically depends on the outcomes collected. While hypothesis testing theory for sequential designs is well established, parameter and precision estimation is less well understood. Even though earlier work has established a number of {\em ad hoc\/} estimators to overcome alleged bias in the ordinary sample average, recent work has shown the sample average to be consistent. Building upon these results, by providing a rate of convergence for the total variation distance, it is established that the asympotic distribution of the sample average is normal, in almost all cases, except in a very specific one where the stopping rule is deterministic and the true population mean coincides with the cut-off between stopping and continuing. For this pathological case, the Kolmogorov distance with the normal is found to equal 0.125. While noticeable in the asymptotic distribution, simulations show that there fortunately are no consequences for the coverage of normally-based confidence intervals.
\end{abstract}

\section{Introduction}
In surprisingly many settings, sample sizes are random. These include sequential trials, clusters of random size, incomplete data, etc. \cite{MKA14} and \cite{MMA16} studied implications of this on estimators in a unified framework; \cite{MMA15} focused on the specific but important case of a sequential trial, which is also the setting of interest in this paper.

While formal sequential methodology dates back to World War II (\cite{W45}), most emphasis has been placed on design and hypothesis testing. Regarding parameter estimation after sequential trials, it has been reported that commonly used estimators, such as the sample average, exhibit bias at least in some settings. In response, a variety of alternative estimators have been proposed (\cite{S78,HP88,EF90}). Building upon \cite{LH99} and \cite{LHW06},  \cite{MKA14}, \cite{MMA16}, and \cite{MMA15} reviewed several of these and actually showed that the sample average is a consistent estimator in spite of earlier concern, even though there is a small amount of finite-sample bias. Their approach is based on considering a class of stochastic stopping rules that lead to the more commonly encountered deterministic stopping rules as limiting cases. They used incomplete-data ignorable likelihood theory to this end. In addition, they showed that there exists an alternative, conditional likelihood estimator that conditions on the sample size realized; this one is unbiased also in small samples but is slightly less efficient than the sample average, and is implicitly defined through an estimating equation.  

While these earlier results are important, the authors did not address the limiting distribution of the mean estimated from a sequential trial and its implications for confidence interval estimation. This is the focus of the current paper. To this end, we consider the manageable but generic case where in a first step $n$ i.i.d.\ normally distributed $N(\mu,1)$ observations are collected, after which a stopping rule is applied and, depending on the outcome, a second i.i.d.\ set of $n$ observations is or is not collected. The probability of stopping after the first round is assumed to be of the form $\Phi\left(\alpha+\frac{\beta}{n}K_n\right)$, with $\Phi(\cdot)$ the probit function, $K_n$ the sample sum of the first $n$ observations, and $\alpha$ and $\beta$ {\em a priori\/} fixed parameters. The setting is formalized in the next section. While there are many cases where other than normal data are collected, it is a sufficiently common and at the same time tractable case; extension to the exponential family is of interest but outside of the scope of this paper. Also for ease of exposition, we consider a study with two possible sample sizes, $n$ and $2n$. Also this can be generalized in a straightforward fashion. Finally, depending on the situation, $K_n/n$ may or may not be the core of the test statistic considered, even though the ratio of the sample sum over a measure of information is very commonly encountered. Calculations for alternative functions of $K_n$ will follow logic similar to the one displayed here. 

Employing the total variation distance, we establish that for stochastic stopping rules asymptotic normality applies. Likewise, we show that this is true too for deterministic stopping rules, provided that $\mu\ne 0$. For these cases rates of convergence are established. When $\mu=0$ there is no weak convergence; we establish the Kolmogorov distance between the true distribution and the normal.

In Section~\ref{framework}, the formal framework is introduced. In Section~\ref{Maintheorem}, the main result is formulated. The behavior in practice is gauged by way of a simulation study, described in Section~\ref{simulations}, with some details relegated to the Appendix. Implications and ramifications are discussed in Section~\ref{sec:Discussion}.

\section{Formal framework\label{framework}}

Let $X_1, X_2, \ldots, X_n, \ldots$ be independent and identically distributed random variables with law $N(\mu,1)$. Also, let $N_1, N_2, \ldots, N_n, \ldots$ be random sample sizes such that each $N_n$ takes the values $n$ or $2n$, is independent of $X_{n+1}, X_{n+2}, \ldots$, and satisfies the conditional law 
\begin{equation}
\mathbb{P}\left[N_n = n \mid X_1, \ldots, X_n\right] = \Phi\left(\alpha + \frac{\beta}{n} K_n\right),\label{eq:PSRule}
\end{equation}
where $\Phi$ is the standard normal cumulative distribution function, $$K_n = \sum_{i=1}^n X_i,$$ $\alpha \in \mathbb{R}$, and $\beta \in \mathbb{R}^+$. Notice that the restriction that $\beta$ be positive is merely for convenience, and that the results presented in this paper can be easily extended for negative $\beta$. We also consider the limiting case of (\ref{eq:PSRule}) where $\beta \rightarrow \infty$, which corresponds to 
\begin{equation}
\mathbb{P}\left[N_n = n \mid X_1, \ldots, X_n\right] = 1_{\left\{K_n > 0\right\}},\label{eq:PSRuleBInfty}
\end{equation}
where $1_{\left\{K_n > 0\right\}}$ stands for the characteristic function of the set $\left\{K_n > 0\right\}$. Finally, we define the estimator
\begin{equation}
\widehat{\mu}_{N_n} = \frac{1}{N_n} K_{N_n},\label{eq:est}
\end{equation}
which is the classical average of a sample with random size $N_n$.

In \cite{MKA14}, it is shown that $\widehat{\mu}_{N_n}$, defined by (\ref{eq:est}), is, for both the stopping rules (\ref{eq:PSRule}) and (\ref{eq:PSRuleBInfty}), a legitimate estimator for $\mu$ in the sense that it is asymptotically unbiased. More precisely, it is established there that, for the probabilistic stopping rule (\ref{eq:PSRule}), 
\begin{equation}
 \mathbb{E}[\widehat{\mu}_{N_n}] = \mu + \frac{1}{2n} \frac{\beta}{\sqrt{1 + \beta^2/n}} \phi\left(\frac{\alpha + \beta \mu}{\sqrt{1 + \beta^2/n}}\right),\label{eq:ExpMu}
\end{equation}
and, for the deterministic stopping rule (\ref{eq:PSRuleBInfty}), 
 \begin{equation}
 \mathbb{E}[\widehat{\mu}_{N_n}] = \mu + \frac{1}{2 \sqrt{n}} \phi(\sqrt{n} \mu),\label{eq:ExpMuBInfty}
 \end{equation}
where $\phi$ is the standard normal density. Clearly, (\ref{eq:ExpMu}) and (\ref{eq:ExpMuBInfty}) both converge to $\mu$ as $n$ tends to $\infty$. These authors also consider small sample bias corrected estimators, but this is outside of the scope of this paper.

In this note, we consider a different aspect of the legitimacy of the estimator $\widehat{\mu}_{N_n}$. More precisely, we examine the asymptotic normality of the sequence 
\begin{equation}
\left(\sqrt{N_n}\left(\widehat{\mu}_{N_n} - \mu\right)\right)_n.\label{eq:BasicSeq}
\end{equation}

\section{Statement of the main result\label{Maintheorem}}

Recall that the Kolmogorov distance between random variables $\xi$ and $\eta$ is given by
$$K(\xi,\eta) = \sup_{x \in \mathbb{R}} \left|\mathbb{P}[\xi \leq x] - \mathbb{P}[\eta \leq x]\right|,$$
and the total variation distance by
$$d_{TV}(\xi,\eta) = \sup_{A} \left|\mathbb{P}[\xi \in A] - \mathbb{P}[\eta \in A]\right|,$$
the supremum running over all Borel sets $A \subset \mathbb{R}$. Clearly, the inequality 
$$K \leq d_{TV}$$
holds, and it is known to be strict in general. Also, it is well known that a sequence of random variables $(\xi_n)_n$ converges weakly to a continuously distributed random variable $\xi$ if and only if $K(\xi,\xi_n) \rightarrow 0$. Finally, $d_{TV}$ metrizes a type of convergence which is in general strictly stronger than weak convergence. For more information on these distances, and on the theory of probability distances in general, we refer the reader to \cite{R91} and \cite{Z83}.

In the following theorem, our main result, we show that if the probabilistic stopping rule (\ref{eq:PSRule}) is followed, then the sequence $\left(\sqrt{N_n}\left(\widehat{\mu}_{N_n} - \mu\right)\right)_n$ converges in total variation distance to $\Phi$, and we establish a rate of convergence in this case. Furthermore, we prove that if the deterministic stopping rule (\ref{eq:PSRuleBInfty}) is followed and $\mu \neq 0$, then the sequence $\left(\sqrt{N_n}\left(\widehat{\mu}_{N_n} - \mu\right)\right)_n$ also converges in total variation distance to $\Phi$, and we again provide a rate of convergence in this case. Finally, we establish that if the deterministic stopping rule (\ref{eq:PSRuleBInfty}) is followed and $\mu = 0$, then, for each $n$, $K(\Phi,\sqrt{N_n}\left(\widehat{\mu}_{N_n} - \mu\right)) = 1/8$. In particular,  $\left(\sqrt{N_n}\left(\widehat{\mu}_{N_n} - \mu\right)\right)_n$ fails to converge weakly to $\Phi$ in this case. We nevertheless show that in all cases it is plausible to use estimation (\ref{eq:est}) for the construction of reliable confidence intervals for $\mu$.

A proof is given in Appendix A.

\begin{thm}\label{thm:CLTCT}
Suppose that the probabilistic stopping rule (\ref{eq:PSRule}) is followed. Then, for each $n$, 
\begin{equation}
d_{TV}(\Phi,\sqrt{N_n}(\widehat{\mu}_{N_n} -  \mu)) \leq C(\alpha,\beta,\mu,n),\label{eq:ROCTV}
\end{equation}
where
\begin{eqnarray*}
\lefteqn{C(\alpha,\beta,\mu,n) = \int_{-\infty}^\infty \phi(u)} \\
&&\left|\Phi\left(\sqrt{\frac{2n}{2n + \beta^2}} (\alpha + \beta \mu) + \frac{\beta}{\sqrt{2 n + \beta^2}} u\right) - \Phi\left(\alpha + \beta \mu + \frac{\beta}{\sqrt{n}} u\right)\right| du,
\end{eqnarray*}
which, by the Dominated Convergence Theorem, converges to $0$ as $n \rightarrow \infty$, whence $\left(\sqrt{N_n}\left(\widehat{\mu}_{N_n} - \mu\right)\right)_n$ converges in total variation distance to $\Phi$. In particular, considering the Borel set $A_x = [-x,x]$ for $x \geq 0$, (\ref{eq:ROCTV}) gives
$$\left|2 \Phi (x) - 1 - \mathbb{P}\left[\widehat{\mu}_{N_n} - \frac{1}{\sqrt{N_n}} x \leq \mu \leq \widehat{\mu}_{N_n} + \frac{1}{\sqrt{N_n}} x \right]\right| \leq C(\alpha,\beta,\mu,n),$$
which makes it plausible to use $\widehat{\mu}_{N_n}$ for the construction of reliable confidence intervals for $\mu$.

Now suppose that the deterministic stopping rule (\ref{eq:PSRuleBInfty}) is followed. Then, for each $n$,
\begin{equation}
d_{TV}(\Phi,\sqrt{N_n}(\widehat{\mu}_{N_n} -  \mu))  \leq C(\mu,n),\label{eq:ROCTVD}
\end{equation}
where
\begin{equation*}
C(\mu,n) =  \int_{-\infty}^\infty \phi(u) \left|1_{\left\{u > -\sqrt{n} \mu\right\}} - \Phi\left(u + \sqrt{2 n } \mu\right)\right| du,
\end{equation*}
which, if $\mu \neq 0$, by the Dominated Convergence Theorem, tends to $0$ as $n \rightarrow \infty$, whence $\left(\sqrt{N_n}\left(\widehat{\mu}_{N_n} - \mu\right)\right)_n$ converges in total variation distance to $\Phi$. In particular, considering the Borel set $A_x = [-x,x]$ for $x \geq 0$, (\ref{eq:ROCTVD}) gives
$$\left|2 \Phi (x) - 1 - \mathbb{P}\left[\widehat{\mu}_{N_n} - \frac{1}{\sqrt{N_n}} x \leq \mu \leq \widehat{\mu}_{N_n} + \frac{1}{\sqrt{N_n}} x \right]\right| \leq C(\mu,n),$$
which, if $\mu \neq 0$, makes it plausible to use $\widehat{\mu}_{N_n}$ for the construction of reliable confidence intervals for $\mu$.

If $\mu = 0$, then, for each $n$,
\begin{equation}
K(\Phi,\sqrt{N_n}\left(\widehat{\mu}_{N_n} - \mu\right)) = 1/8,\label{eq:KBad}
\end{equation}
and $\left(\sqrt{N_n}\left(\widehat{\mu}_{N_n} - \mu\right)\right)_n$ fails to converge weakly to $\Phi$. Nevertheless, for each $x \in \mathbb{R}^+_0$,
\begin{equation}
\mathbb{P}\left[\widehat{\mu}_{N_n} - \frac{1}{\sqrt{N_n}} x \leq \mu \leq \widehat{\mu}_{N_n} + \frac{1}{\sqrt{N_n}} x \right] = 2 \Phi(x) - 1.\label{eq:CImuzero}
\end{equation}
Thus, also in the case where $\mu = 0$, it is plausible to use $\widehat{\mu}_{N_n}$ for the construction of reliable confidence intervals for $\mu$.
\end{thm}

\section{Simulations\label{simulations}}
We have conducted a brief simulation study to illustrate Theorem \ref{thm:CLTCT}, the tables of which are given in Appendix B. We have studied the empirical distribution $\mathcal{E}_n$ of $\sqrt{N_n}(\widehat{\mu}_{N_n} - \mu)$, based on 1000 simulations, for both the probabilistic stopping rule (\ref{eq:PSRule}) (Tables 1 and 2) and the deterministic stopping rule (\ref{eq:PSRuleBInfty}) (Table 3), and different values for $\beta$, the true parameter $\mu$, and the number of observations $n$. In each case, we have compared the theoretical upper bound for the total variation distance between the standard normal distribution and the theoretical distribution of $\sqrt{N_n}(\widehat{\mu}_{N_n} - \mu)$, as given in Theorem \ref{thm:CLTCT}, with the Kolmogorov distance between the standard normal cdf and the empirical distribution of  $\sqrt{N_n}(\widehat{\mu}_{N_n} - \mu)$. We have also counted the number of times out of 1000 where the true parameter $\mu$ is contained in the interval $\left[\widehat{\mu}_{N_n} - 1.96/\sqrt{N_n} , \widehat{\mu}_{N_n} + 1.96/\sqrt{N_n} \right]$, which would be a $95 \%$-confidence interval for $\mu$ if $\sqrt{N_n}(\widehat{\mu}_{N_n} - \mu)$ were standard normally distributed.

The predictions by Theorem \ref{thm:CLTCT} are confirmed by the simulation study. More precisely, in the cases where the stopping rule is close to being deterministic and $\mu = 0$, the simulation study indeed points out that the distribution of $\sqrt{N_n}(\widehat{\mu}_{N_n} - \mu)$ deviates from a standard normal distribution (red values in the tables). However, it is also confirmed that for the construction of confidence intervals for $\mu$, it is `harmless' to nevertheless assume that $\sqrt{N_n}(\widehat{\mu}_{N_n} - \mu)$ is standard normally distributed.

\section{Discussion}\label{sec:Discussion}
While sequential designs are in common use in medical and other applications, and while the hypothesis testing theory based there upon has been well established for a long time, there is more confusion about parameter and precision estimation following such a sequential study.  \cite{MKA14}, \cite{MMA16}, and \cite{MMA15} showed that the sample average is a valid estimator, with both stochastic and deterministic stopping rules, for a wide class of normal and exponential-family-based models. They established that this estimator, in spite of small-sample bias and the fact that there is no uniform minimum-variance unbiased estimator, is consistent and hence asymptotically unbiased. 

Building upon this work, in this paper, we have shown that the sample average in the case of normally distributed outcomes is also asymptotically normal in a broad range of situations. First, this is true with stochastic stopping rule. Second, it applies in almost all deterministic stopping rule situations within the class considered, except in the very specific case where the normal population mean $\mu=0$. Note that the special status of the null value stems from the fact that the cut-off  between stopping and continuing associated with our deterministic stopping rule is equal to zero. It can easily be shown, should the cut-off point be shifted to a non-zero value, that then the problematic value for $\mu$ also shifts.

We also showed that the Kolmogorov distance, for $\mu=0$, equals 
$$K(\Phi,\sqrt{N_n}\left(\widehat{\mu}_{N_n} - \mu\right)) = 1/8,$$ 
from which it follows that $\left(\sqrt{N_n}\left(\widehat{\mu}_{N_n} - \mu\right)\right)_n$ does not converge weakly to $\Phi$ in this case. It is enlightening that the qualitative non-convergence result is supplemented with a quantitative determination of the deviation from normality. 

To further examine the extent of the result obtained, simulations show that, indeed, asymptotic normality becomes more problematic when $\mu$ approaches zero and the parameter $\beta$ approaches $+\infty$, with the latter value corresponding to a deterministic rule. However, asymptotic normality is invoked predominantly to calculate normally based confidence intervals. It is therefore very reassuring that using such intervals for $\mu=0$ and a deterministic stopping rule does not lead to any noticeable effect on the coverage probabilities. 

In summary, we can conclude that for relevant classes of stopping rules, the sample average and corresponding normal confidence interval can be used without problem. It will be of interest to examine in more detail the situation of outcomes that follow an exponential family distribution, other than the normal one.

\newpage
\appendix

\section{Proof of Theorem \ref{thm:CLTCT}}\label{sec:PCLTCT}

Before writing down the proof of Theorem \ref{thm:CLTCT}, we give three lemmas. Part of Lemma \ref{lem:jd} can be found in \cite{MKA14}, but as it belongs to the heart of our calculations, we present a complete proof here. 

\begin{lem}\label{lem:lem1}
For $A,B \in \mathbb{R}$,
\begin{equation}
\int_{-\infty}^\infty \phi(x) \Phi(A + B x) dx = \Phi\left(\frac{A}{\sqrt{1 + B^2}}\right).\label{eq:GInt1}
\end{equation}
\end{lem}

\begin{proof}
This is standard.
\end{proof}

\begin{lem}\label{lem:lem2}
For $k,z \in \mathbb{R}$,
\begin{equation}
\phi\left(\frac{z - n \mu}{\sqrt{n}}\right)  \phi\left(\frac{k - z - n \mu}{\sqrt{n}}\right) = \phi\left(\frac{k - 2 n \mu}{\sqrt{2 n}}\right) \phi\left(\frac{2 z - k}{\sqrt{2 n}}\right).\label{eq:2Gpdf}
\end{equation}
\end{lem}

\begin{proof}
This follows by a straightforward calculation.
\end{proof}

\begin{lem}\label{lem:jd}
Let $f_{N_n,K_{N_n}}$ be the joint density of $N_n$ and $K_{N_n}$. Then, for the probabilistic stopping rule (\ref{eq:PSRule}), 
\begin{equation}
f_{N_n,K_{N_n}}(n,k) = \frac{1}{\sqrt{n}} \phi\left(\frac{k - n \mu}{\sqrt{n}}\right) \Phi\left(\alpha + \frac{\beta k}{n} \right)\label{eq:jdnps}
\end{equation}
and
\begin{equation}
f_{N_n,K_{N_n}}(2n, k) = \frac{1}{\sqrt{2n}} \phi\left(\frac{k - 2n\mu}{\sqrt{2n}}\right) \left[1 - \Phi\left(\frac{\alpha + \frac{\beta k }{2n}}{\sqrt{\frac{2n + \beta^2}{2n}}}\right)\right],\label{eq:jd2nps}
\end{equation}
and, for the deterministic stopping rule (\ref{eq:PSRuleBInfty}),
\begin{equation}
f_{N_n,K_{N_n}}(n,k) = \frac{1}{\sqrt{n}} \phi\left(\frac{k - n \mu}{\sqrt{n}}\right) 1_{\{k > 0\}}\label{eq:jdnds}
\end{equation}
and
\begin{equation}
f_{N_n,K_{N_n}}(2n, k) = \frac{1}{\sqrt{2n}} \phi\left(\frac{k - 2n\mu}{\sqrt{2n}}\right) \left[1 - \Phi\left(\frac{k}{\sqrt{2n}}\right)\right].\label{eq:jd2nds}
\end{equation}
\end{lem}

\begin{proof}
First suppose that the probabilistic stopping rule (\ref{eq:PSRule}) is followed. 

Notice that
\begin{equation}
f_{N_n,K_{N_n}}(n,k) = f_{N_n,K_{n}}(n,k) = f_{K_n}(k) f_{N_n \mid K_n}(n\mid k),\label{eq:Bayes1}
\end{equation}
with $f_{K_n}$ the density of $K_n$, and $f_{N_n \mid K_n}$ the conditional density of $N_n$ given $K_n$. Now, the $X_i$ being independent and normally distributed with mean $\mu$ and variance $1$, we have
\begin{equation}
f_{K_n}(k) = \frac{1}{\sqrt{n}} \phi\left(\frac{k - n  \mu}{\sqrt{n}}\right).\label{eq:densityKn}
\end{equation}
Furthermore, by (\ref{eq:PSRule}), 
\begin{equation}
f_{N_n \mid K_n}(n\mid k) = \Phi\left(\alpha + \frac{\beta k}{n}\right).\label{eq:condensNnKn}
\end{equation}
Combining (\ref{eq:Bayes1}), (\ref{eq:densityKn}), and (\ref{eq:condensNnKn}), establishes (\ref{eq:jdnps}). 

We now establish (\ref{eq:jd2nps}). Observe that
\begin{eqnarray}
f_{N_n,K_n}(2n,k)
&=& f_{N_n,K_{2n}}(2n,k)\nonumber\\
&=& f_{K_{2n}}(k) - f_{N_n,K_{2n}}(n,k)\nonumber\\
&=& f_{K_{2n}}(k) - \left(f_{N_n,K_n}(n,\cdot) \star f_{\sum_{i = n+1}^{2n} X_i}\right)(k),\label{eq:Bayes2}
\end{eqnarray}
$\star$ being the convolution product, and the last equality following by indepence of $N_n$ and $X_{n+1}, \ldots, X_{2n}$. Using (\ref{eq:jdnps}) and the fact that the $X_i$ are independent and normally distributed with mean $\mu$ and variance $1$, (\ref{eq:Bayes2}) equals
\begin{equation*}
\frac{1}{\sqrt{2n}}\phi\left(\frac{k - 2n \mu}{\sqrt{2n}}\right) - \frac{1}{n} \int_{-\infty}^\infty \phi\left(\frac{z - n \mu}{\sqrt{n}}\right) \Phi\left(\alpha +  \frac{\beta z}{n}\right) \phi\left(\frac{k - z - n \mu}{\sqrt{n}} \right) dz,
\end{equation*}
which, by (\ref{eq:2Gpdf}), 
\begin{equation}
= \frac{1}{\sqrt{2n}} \phi\left(\frac{k - 2n \mu}{\sqrt{2n}}\right) - \frac{1}{n} \phi\left(\frac{k - 2 n \mu}{\sqrt{2 n}}\right) \int_{-\infty}^\infty \phi\left(\frac{2 z - k}{\sqrt{2n}}\right) \Phi\left(\alpha + \frac{\beta z}{n}\right) dz.\label{eq:Bayes3}
\end{equation}
After performing the change of variables $u = \frac{2 z - k}{\sqrt{2n}}$, (\ref{eq:Bayes3}) reduces to
\begin{equation*}
\frac{1}{\sqrt{2n}} \phi\left(\frac{k - 2n \mu}{\sqrt{2n}}\right) \left[1 - \int_{- \infty}^\infty  \phi(u) \Phi\left(\alpha + \frac{\beta k}{2n} + \frac{\beta}{\sqrt{2n}} u \right) du\right],\label{eq:Bayes4}
\end{equation*}
which, by (\ref{eq:GInt1}), coincides with
\begin{equation*}
\frac{1}{\sqrt{2n}} \phi\left(\frac{k - 2n \mu}{\sqrt{2n}}\right) \left[1 - \Phi\left(\frac{\alpha + \frac{\beta k}{2n}}{\sqrt{1 + \frac{\beta^2}{2n}}}\right)\right].
\end{equation*}
This proves that (\ref{eq:jd2nps}) holds.

Now suppose that the detereministic stopping rule (\ref{eq:PSRuleBInfty}) is followed. 

Of course, (\ref{eq:Bayes1}) and (\ref{eq:densityKn}) continue to hold, and, by (\ref{eq:PSRuleBInfty}),
$$f_{N_n\mid K_n}(n\mid k) = 1_{\left\{k > 0\right\}},$$
from which (\ref{eq:jdnds}) follows. 

To establish (\ref{eq:jd2nds}), notice that (\ref{eq:Bayes3}) also continues to hold, which, by (\ref{eq:jdnds}), gives
\begin{equation*}
f_{N_n,K_n}(2n,k) = \frac{1}{\sqrt{2n}}\phi\left(\frac{k - 2n \mu}{\sqrt{2 n}}\right) - \frac{1}{n} \int_0^\infty \phi\left(\frac{z - n \mu}{\sqrt{n}}\right) \phi\left(\frac{k - z - n \mu}{\sqrt{n}}\right) dz,
\end{equation*}
which, by (\ref{eq:2Gpdf}),
\begin{equation*}
= \frac{1}{\sqrt{2n}} \phi\left(\frac{k - 2 n \mu}{\sqrt{2n}}\right) - \frac{1}{n} \phi\left(\frac{k - 2 n \mu}{\sqrt{2n}}\right) \int_0^\infty \phi\left(\frac{2 z - k}{\sqrt{2 n}}\right) dz,
\end{equation*}
which, after performing the change of variables $u = \frac{2 z - k}{\sqrt{2n}}$, 
\begin{eqnarray*}
&=& \frac{1}{\sqrt{2 n}} \phi\left(\frac{k - 2 n \mu}{\sqrt{ 2n}}\right) \left[1 - \int_{- k/\sqrt{2n}}^\infty \phi(u) du\right]\\
&=& \frac{1}{\sqrt{2 n}} \phi\left(\frac{k - 2 n \mu}{\sqrt{ 2n}}\right) \left[1 - \Phi\left(\frac{k}{\sqrt{2n}}\right)\right].
\end{eqnarray*}
This finishes the proof of (\ref{eq:jd2nds}).
\end{proof}

\begin{proof}[Proof of Theorem \ref{thm:CLTCT}]
First suppose that the probabilistic stopping rule (\ref{eq:PSRule}) is followed. 

For $n$ and a Borel set $A \subset \mathbb{R}$,
\begin{eqnarray}
\lefteqn{\mathbb{P}\left[\sqrt{N_n} \left(\widehat{\mu}_{N_n} - \mu\right) \in A\right]}\nonumber\\
&=& \mathbb{P}\left[\frac{K_{N_n} - N_n \mu}{\sqrt{N_n}} \in A\right]\nonumber\\
&=& \mathbb{P}\left[\frac{K_{n} - n \mu}{\sqrt{n}} \in A, N_n = n\right] + \mathbb{P}\left[\frac{K_{2n} - 2n \mu}{\sqrt{2n}} \in A, N_n = 2n\right].\label{eq:SplitUpn2n} 
\end{eqnarray}
Plugging in (\ref{eq:jdnps}) and (\ref{eq:jd2nps}) in (\ref{eq:SplitUpn2n}), gives
\begin{equation}
\mathbb{P}\left[\sqrt{N_n} \left(\widehat{\mu}_{N_n} - \mu\right) \in A\right] = I_1 + I_2,\label{eq:EhI1I2}
\end{equation}
with
\begin{equation*}
I_1 = \int 1_A\left(\frac{k - n \mu}{\sqrt{n}}\right)  \frac{1}{\sqrt{n}} \phi\left(\frac{k - n \mu}{\sqrt{n}}\right) \Phi\left(\alpha + \frac{\beta k }{n}\right) dk
\end{equation*}
and
\begin{equation*}
I_2 = \int 1_A\left(\frac{k - 2 n \mu}{\sqrt{2 n}}\right) \frac{1}{\sqrt{2 n}} \phi\left(\frac{k - 2 n \mu}{\sqrt{2 n}}\right) \left[1 - \Phi\left(\frac{\alpha + \frac{\beta k}{2n}}{\sqrt{\frac{2n + \beta^2}{2n}}}\right)\right] dk.
\end{equation*}
Performing the change of variables $u = \frac{k - n \mu}{\sqrt{n}}$ shows that 
\begin{equation}
I_1 = \int_{A}  \phi(u) \Phi\left(\alpha + \beta \mu + \frac{\beta}{\sqrt{n}} u\right) du,\label{eq:I1el}
\end{equation}
and performing the change of variables $u = \frac{k - 2 n \mu}{\sqrt{2n}}$ gives
\begin{equation}
I_2 = \int_{A} \phi(u) \left[1 - \Phi\left(\sqrt{\frac{2n}{2n+\beta^2}} (\alpha + \beta \mu) + \frac{\beta}{\sqrt{2 n + \beta^2}} u\right)\right] du.\label{eq:I2el}
\end{equation}
Combining (\ref{eq:EhI1I2}), (\ref{eq:I1el}), and (\ref{eq:I2el}), yields (\ref{eq:ROCTV}).

Now suppose that the deterministic stopping rule (\ref{eq:PSRuleBInfty}) is followed. Fix $n$ and a Borel set $A \subset \mathbb{R}$. Of course, (\ref{eq:SplitUpn2n}) continues to hold. Plugging in (\ref{eq:jdnds}) and (\ref{eq:jd2nds}) in (\ref{eq:SplitUpn2n}) gives
\begin{equation}
\mathbb{P}\left[\sqrt{N_n}\left(\widehat{\mu}_{N_n} - \mu\right) \in A\right] = L_1 + L_2,\label{eq:SplitUpL1L2}
\end{equation}
with
\begin{equation*}
L_1 = \int_{-\infty}^\infty 1_A\left(\frac{k - n \mu}{\sqrt{n}}\right) \frac{1}{\sqrt{n}} \phi\left(\frac{k - n \mu}{\sqrt{n}}\right) 1_{\{k > 0\}} dk
\end{equation*}
and
\begin{equation*}
L_2 = \int_{-\infty}^\infty 1_A\left(\frac{k - 2 n \mu}{\sqrt{2n}}\right) \frac{1}{\sqrt{2n}} \phi\left(\frac{k - 2 n \mu}{\sqrt{2n}}\right) \left[1 - \Phi\left(\frac{k}{\sqrt{2n}}\right)\right] dk.
\end{equation*}
Performing the change of variables $u = \frac{k - n \mu}{\sqrt{n}}$ leads to
\begin{equation}
L_1 = \int_A \phi(u) 1_{\{u > - \sqrt{n} \mu\}}du,\label{eq:L1el}
\end{equation}
and performing the change of variables $u = \frac{k - 2 n \mu}{\sqrt{2n}}$ yields
\begin{equation}
L_2 = \int_{A} \phi(u) \left[1 - \Phi\left(u + \sqrt{2n} \mu\right)\right] du.\label{eq:L2el}
\end{equation}
Now (\ref{eq:SplitUpL1L2}), (\ref{eq:L1el}), and (\ref{eq:L2el}) give (\ref{eq:ROCTVD}).

We now turn to the case $\mu = 0$. Replacing $A$ by $\left]-\infty,x\right]$ in (\ref{eq:SplitUpL1L2}), (\ref{eq:L1el}), and (\ref{eq:L2el}), shows that, for $x \geq 0$,
\begin{eqnarray*}
\left|\Phi(x) - \mathbb{P}[\sqrt{N_n} (\widehat{\mu}_{N_n} - \mu) \leq x]\right| &=&  \left|\int_{-\infty}^x \phi(u) \Phi(u) du - \int_0^x \phi(u) du\right|\\
&=& \left|\left[\Phi(x)\right]^2/2 - \Phi(x) + 1/2\right|,
\end{eqnarray*}
which assumes the maximal value $1/8$ on $[0,\infty[$, and, for $x \leq 0$,
\begin{eqnarray*}
\left|\Phi(x) - \mathbb{P}[\sqrt{N_n} (\widehat{\mu}_{N_n} - \mu) \leq x]\right| &=& \int_{-\infty}^x \phi(u) \Phi(u) du\\
&=& \left[\Phi(x)\right]^2/2,
\end{eqnarray*}
which assumes the maximal value $1/8$ on $]-\infty,0]$. This proves (\ref{eq:KBad}). Finally, replacing $A$ by $\left[-x,x\right]$ in (\ref{eq:SplitUpL1L2}), (\ref{eq:L1el}), and (\ref{eq:L2el}), shows that, for $x \geq 0$,
\begin{eqnarray*}
\lefteqn{\left|2\Phi(x) - 1 - \mathbb{P}[-x \leq \sqrt{N_n} (\widehat{\mu}_{N_n} - \mu) \leq x]\right|}\\
&=&  \left|\int_{-x}^x \phi(u) \Phi(u) du - \int_0^x \phi(u) du\right|\\
&=& \left|\left[\Phi(x)\right]^2/2 - \left[\Phi(-x)\right]^2/2 - \Phi(x) + 1/2\right|\\
&=& 0,
\end{eqnarray*}
which proves (\ref{eq:CImuzero}).
\end{proof}

\newpage
\section{Tables from the simulation study}

\begin{table}[H]\label{t:betasmall}
\caption{\small{Simulation study for estimation (\ref{eq:est}) for the probabilistic stopping rule (\ref{eq:PSRule}); $\alpha =0$; $\beta$ small; $\mu$, true mean for the standard normal from which the sample is taken; $n$, number of observations; $C = C(\alpha,\beta,\mu,n)$ in Theorem \ref{thm:CLTCT}; $K$, the Kolmogorov distance between the standard normal cdf and the empirical cdf of $\sqrt{N_n}(\widehat{\mu}_{N_n} - \mu)$} based on 1000 simulations; $L$, number of times out of 1000 where the true parameter $\mu$ is contained in the interval $\left[\widehat{\mu}_{N_n} - 1.96/\sqrt{N_n} , \widehat{\mu}_{N_n} + 1.96/\sqrt{N_n} \right]$ (which would be a $95 \%$-confidence interval if $\sqrt{N_n}(\widehat{\mu}_n - \mu)$ were standard normally distributed).}
\begin{center}
    \begin{tabular}{ | r | r | r | r | r | r ||  r | r | r | r | r | r |}
    
    \hline
    
    $\beta$ & $\mu$ & $n$ & $C$ & $K$ & $L$ & $\beta$ & $\mu$ & $n$ & $C$ & $K$ & $L$ \\ 
    
    \hline
   
    0 & -10 & 10 & 0.000&0.024 & 947 &1 & -10 & 10 &0.000& 0.015 &958\\
    0 & -10 & 100 & 0.000&0.018 & 948&1 & -10 & 100 &0.000& 0.018&947 \\
    0 & -10 & 1000 &0.000& 0.021 &941& 1 & -10 & 1000 &0.000& 0.029 &949
\\

    0 & -1 & 10 &0.000 &0.026 &947 &1 & -1 & 10 & 0.081& 0.024&960 \\

    0 & -1 & 100 &0.000 &0.030 &  948&1 & -1 & 100 & 0.025&0.047&954 \\

    0 & -1 & 1000 &0.000 &0.021 & 952&    1 & -1 & 1000 &0.008  &0.014& 941\\

    0 & 0& 10 & 0.000&0.021  &941 &1 & 0& 10 &0.120 & 0.042&950 \\

    0 & 0& 100 & 0.000&0.030& 953 &1 & 0& 100 &0.039& 0.017&952 \\

    0 & 0 & 1000 &0.000& 0.011& 958 &   1 & 0 & 1000 &0.013& 0.012&954\\

    0 & 1 & 10 &  0.000&0.027 & 954 & 1 & 1 & 10 &0.071& 0.026&  957\\

    0 & 1 & 100 & 0.000&0.017 &957&     1 & 1 & 100 & 0.023&0.016&  955\\

    0 & 1 & 1000 & 0.000&0.045 &957&    1 & 1 & 1000 &0.008& 0.036& 941\\
    0 & 10 & 10 &0.000&0.039   & 950  &1 & 10 & 10 &0.000& 0.037&951 \\

    0 & 10& 100 & 0.000&0.026 & 943& 1 & 10& 100 &0.000& 0.026& 952\\

    0 & 10 & 1000 &0.000& 0.024 &956& 1 & 10 & 1000 & 0.000 &0.028& 949\\
	\hline
   \end{tabular}
   
   \end{center}
   
   \end{table}
   \begin{table}[H]
   \caption{\small Same setup as in Table 1. Now $\beta$ is moderately large. }

      \begin{center}

   \begin{tabular}{  | r | r | r | r | r | r ||  r | r | r | r | r | r |}
    \hline
    
     $\beta$ & $\mu$ & $n$ & $C$ & $K$ & $L$ &  $\beta$ & $\mu$ & $n$ & $C$ & $K$ & $L$\\ 
    
    \hline
   
    10 & -10 & 10 &0.000& 0.010&949&100 & -10 & 10 &0.000& 0.019 &948\\

    10 & -10 & 100 &0.000& 0.020&955&100 & -10 & 100 &0.000& 0.024& 944\\

    10 & -10 & 1000 &0.000& 0.024 &952& 100 & -10 & 1000 &0.000& 0.023&946\\

    10 & -1 & 10 &0.002& 0.015 &953&100 & -1 & 10 &0.001& 0.039&960\\
    10 & -1 & 100 &0.000& 0.017 &955&100 & -1 & 100 &0.000 & 0.015& 953\\

    10 & -1 & 1000 &0.000& 0.017& 947&  100 & -1 & 1000 &0.000& 0.011&941\\

    10 & 0& 10 &0.440& \color{red} 0.084 &945&100 & 0& 10 &0.494& \color{red} 0.145&950\\

    10 & 0& 100 &0.300& 0.021&948  & 100 & 0& 100 &0.481& \color{red} 0.080&948\\
    10 & 0 & 1000 &0.120& 0.047&950&100 & 0 & 1000 &0.437& \color{red} 0.068&946\\

    10 & 1 & 10 &0.001& 0.028&942&100 & 1 & 10 &0.001& 0.035&946\\
    10 & 1 & 100 &0.000& 0.026 &973&100 & 1 & 100 &0.000 &  0.011&938\\

    10 & 1 & 1000 &0.000& 0.019 &940& 100 & 1 & 1000 &0.000&  0.021& 951\\

    10 & 10 & 10 &0,000& 0.021&  967& 100 & 10 & 10 &0.000&  0.009&954\\

    10 & 10& 100 &0.000& 0.009&  955&  100 & 10& 100 &0,000&  0.014&960\\

    10 & 10 & 1000 &0.000& 0.033&948 &  100 & 10 & 1000 &0.000& 0.010&954\\

   \hline
\end{tabular}
\end{center}

\end{table}
\begin{table}[H]
\caption{\small{Simulation study for estimation (\ref{eq:est}) for the deterministic stopping rule (\ref{eq:PSRuleBInfty}); $\mu$, true mean for the standard normal from which the sample is taken; $n$, number of observations; $C = C(\mu,n)$ in Theorem \ref{thm:CLTCT}; $K$, the Kolmogorov distance between the standard normal cdf and the empirical cdf of $\sqrt{N_n}(\widehat{\mu}_{N_n} - \mu)$} based on 1000 simulations; $L$, number of times out of 1000 where the true parameter $\mu$ is contained in the interval $\left[\widehat{\mu}_{N_n} - 1.96/\sqrt{N_n} , \widehat{\mu}_{N_n} + 1.96/\sqrt{N_n} \right]$ (which would be a $95 \%$-confidence interval if $\sqrt{N_n}(\widehat{\mu}_n - \mu)$ were standard normally distributed).}
\begin{center}
\begin{tabular}{|r | r | r | r | r |r|}
    \hline
    
    $\beta$ & $\mu$ & $n$ & $C$ & $K$& $L$ \\ 
    
    \hline
   
    $\infty$ & -10 & 10 &0.000& 0.005&943 \\
    $\infty$ & -10 & 100 &0.000& 0.026& 949\\
    $\infty$ & -10 & 1000 &0.000& 0.029&949\\
    $\infty$ & -1 & 10 &0.002 & 0.034&952\\
    $\infty$ & -1 & 100 &0.000& 0.025&956\\
    $\infty$ & -1 & 1000 &0.000& 0.018&959\\
    $\infty$ & 0& 10 &0.250& \color{red} 0.130&940\\
    $\infty$ & 0& 100 &0.250& \color{red} 0.113&941\\
    $\infty$ & 0 & 1000 &0.250& \color{red} 0.129&958\\
    $\infty$ & 1 & 10 &0.002& 0.022&955\\
    $\infty$ & 1 & 100 &0.000& 0.034&968\\
    $\infty$ & 1 & 1000 &0,000&  0.005&964\\
    $\infty$ & 10 & 10 &0.000& 0.016&946 \\
    $\infty$ & 10& 100 &0,000&  0.021&953\\
    $\infty$ & 10 & 1000 &0,000&  0.023&952\\

   \hline
   \end{tabular} 
   
   \end{center}
   
   \end{table}

\end{document}